\documentclass[12pt]{amsart}

\usepackage[cp1250]{inputenc}
\usepackage{graphicx}
\usepackage{url}
\usepackage[all]{xy}
\usepackage{multicol}
\usepackage{amsthm}
\usepackage{amsmath}
\usepackage{amssymb}
\usepackage{geometry}
\usepackage{pinlabel}
\usepackage{hyperref}
\usepackage{array}
\usepackage{xcolor}

\newtheorem{Theorem}{Theorem}[section]
\newtheorem{lemma}[Theorem]{Lemma}
\newtheorem{proposition}[Theorem]{Proposition}

\newtheorem{definition}[Theorem]{Definition}
\newtheorem{remark}[Theorem]{Remark}
\newtheorem{example}[Theorem]{Example}

\newcommand{\ZZ}{\mathbb {Z}}

\newcommand{\tr}{\triangleright}

\newcommand{\ho}{\textrm{Hom}_{B}}
\newcommand{\hoq}{\textrm{Hom}_{Q}}
\newcommand{\Ho}{\textbf{Hom}_{B}}
\newcommand{\Hoq}{\textbf{Hom}_{Q}}
\newcommand{\q}{\mathcal{Q}}

\def \dow {\underline{\ast }}
\def \up {\overline{\ast }}

\begin{document}
\title{From biquandle structures to Hom-biquandles}

\author{Eva Horvat, Alissa S. Crans}
\address{University of Ljubljana\\
Faculty of Education\\
Kardeljeva plo\v s\v cad 16\\
1000 Ljubljana, Slovenia}
\email{eva.horvat@pef.uni-lj.si}
\address{Loyola Marymount University \\
Department of Mathematics \\
One LMU Drive, Suite 2700\\
Los Angeles, CA 90045\\
USA}
\email{acrans@lmu.edu}

\date{\today}
\maketitle

\begin{abstract}
We investigate the relationship between the quandle and biquandle coloring invariant and obtain an enhancement of the quandle and biquandle coloring invariants using biquandle structures.  

We also continue the study of biquandle homomorphisms into a medial biquandle begun in \cite{CR1}, finding biquandle analogs of results therein.  We describe the biquandle structure of the Hom-biquandle, and consider the relationship between the Hom-quandle and Hom-biquandle.
\end{abstract}

\section {Introduction}
\label{sec0}
Quandles and their generalizations, biquandles, are algebraic structures whose axioms encode the Reidemeister, and oriented Reidemeister, moves from classical knot theory.  Biquandle invariants provide a method for distinguishing between certain virtual (and some non-virtual) knots.  In this article, we study the relationship between quandles and biquandles, with the goal of finding biquandle versions of results pertaining to sets of quandle homomorphisms.

We begin in Section \ref{sec1} with a brief review of basic quandle and biquandle definitions and facts together with fundamental examples.  In Section \ref{sec2} we recall the notion of a biquandle structure introduced in \cite{EH} and provide examples of different such structures one can place on the same quandle.  We further present properties that a biquandle does, and does not, inherit from its associated quandle. We consider mediality and commutativity of biquandles and biquandle structures. We turn our focus to connections with knot theory in Section \ref{sec3} by exploring the relationship between the quandle and biquandle coloring invariant, illustrating this with two concrete examples that demonstrate how the richness of biquandle structures on a given quandle can improve the strength of (bi)quandle representation invariants. We define an enhancement of quandle and biquandle coloring invariants based on biquandle structures. In Section \ref{sec4} we continue the study of biquandle homomorphisms into a medial biquandle begun in \cite{CR1}, finding biquandle analogs of results therein.  We describe the biquandle structure of the Hom-biquandle, and consider the relationship between the Hom-quandle and Hom-biquandle, adding some sample calculations.  We conclude in Section \ref{sec5} with questions for future investigation.

\section{Preliminaries}
\label{sec1}

We begin by recalling definitions and examples of quandles and biquandles.  We refer the reader to \cite{CAR, JO, FR, FJK} for more details.

\begin{definition}\label{def1} A \textbf{quandle} is a set $Q$ equipped with a binary operation $*\colon Q\times Q\to Q$ that satisfies the following three axioms:
\begin{itemize}
\item $x*x=x$ for every $x\in Q$; 
\item the map $R _y\colon Q\to Q$ given by $R _y (x)=x*y$ is a bijection for every $y\in Q$; and
\item $(x*y)*z=(x*z)*(y*z)$ for every $x,y,z\in Q$. 
\end{itemize} 
\end{definition}
On any set $Q$, we can define a \textbf{trivial quandle} using the binary operation $x*y=x$ for every $x,y\in Q$.  Given two quandles $(Q,*)$ and $(K,\circ )$, a map $f\colon Q\to K$ is called a \textbf{quandle homomorphism} if $f(x*y)=f(x)\circ f(y)$ for every $x,y\in Q$. Note that the third axiom above implies that each $R_y$ is a quandle homomorphism.\\
 
 Two particularly important examples of quandles include the following.
 
 \begin{itemize}
 \item[(a)] On any group $G$ we may define a quandle operation by $g*h=hg^{-1}h$ for any $g,h\in G$.  This gives what is known as the \textbf{core quandle,} $\textit{Core}\,(G)$.
 \item[(b)] For an abelian group $G$ and a chosen group automorphism $\phi \in Aut(G)$, the binary operation $g*h=\phi (g)+(1-\phi )(h)$ defines an \textbf{affine quandle,} $\textit{Aff}\,(G,\phi )$.  
 \end{itemize}
We can generalize the notion of a quandle as follows: 
\begin{definition}\label{def2} A \textbf{biquandle} is a set $X$ with two binary operations $\dow , \up \colon X\times X\to X$ that satisfy the following axioms:
\begin{itemize}
\item $x\dow x=x\up x$ for every $x\in X$; 
\item the maps $\alpha _y, \beta _y \colon X\to X$ and $S\colon X\times X\to X\times X$, given by $\alpha _y (x)=x\dow y$, $\beta _y (x)=x\up y$ and $S(x,y)=(y\up x,x\dow y)$ are bijections for every $y\in X$; and
\item the exchange laws \begin{xalignat*}{1}
& (x\dow y)\dow (z\dow y)=(x\dow z)\dow (y\up z)\\
& (x\dow y)\up (z\dow y)=(x\up z)\dow (y\up z) \\
& (x\up y)\up (z\up y)=(x\up z)\up (y\dow z)
\end{xalignat*} hold for every $x,y,z\in X$. 
\end{itemize} 
\end{definition}

We note that if $x\up y=x$ for any $x,y\in (X,\dow ,\up )$, then $(X,\dow )$ is a quandle. Thus biquandles are a generalization of quandles. In fact, any biquandle $(X,\dow ,\up )$ has an \textbf{associated quandle,} $\q (X)=(X,*)$, defined by the operation $x*y=\beta _{y}^{-1}(x\dow y)$, and this induces a functor $\mathcal{Q}$ from the category of biquandles to the category of quandles \cite{ASH}, \cite{ISH}, \cite[Lemma 3.1]{EH}.

The biquandle analogs of our quandle examples are:
\begin{itemize} 
\item[($\mathrm{a}'$)] For any group $G$, the binary operations $g\dow h=h^{-1}g^{-1}h$ and $g\up h=h^{-2}g$ define a biquandle that is called the {\it Wada biquandle}. Its associated quandle is the core quandle $\textit{Core}\,(G)$. 
\item[($\mathrm{b}'$)] Let $G$ be an abelian group and choose two automorphisms $\phi ,\psi \in Aut(G)$. The operations $g\dow h=\psi (\phi (g))+(\psi -\psi \phi )(h)$ and $g\up h=\psi (g)$ define a biquandle whose associated quandle is the affine quandle $\textit{Aff}\,(G,\phi )$. 
\end{itemize}
In addition, given two quandles $(Q,*)$ and $(K,\circ )$, the product $Q\times K$ is a biquandle with the operations $(x,y)\dow (z,w)=(x*z,y)$ and $(x,y)\up (z,w)=(x,y\circ ^{-1}w)$.

Given two biquandles $(X,\dow ,\up )$ and $(Z,\veebar ,\barwedge)$, a map $f\colon X\to Z$ is called a \textbf{biquandle homomorphism} if $f(x\dow y)=f(x)\veebar f(y)$ and $f(x\up y)=f(x)\barwedge f(y)$ for every $x,y\in X$.

\section{Biquandle structures on quandles}
\label{sec2}

The algebraic structures of quandles and biquandles are closely intertwined. As mentioned in Section \ref{sec1}, every biquandle $X$ has an associated quandle $\q (X)$. On the other hand, on a given quandle we may impose several nonequivalent structures that define biquandles. In this Section, we present the notion of a `biquandle structure' and discuss which properties of biquandles are inherited from the properties of their associated quandle.

\begin{definition}\label{def3}  Let $(Q,*)$ be a quandle. A \textbf{biquandle structure} on $(Q,*)$ is a family of quandle automorphisms $\{\beta _{y}\colon Q\to Q|\, y\in Q\}\subset Aut(Q)$ that satisfies the following conditions:
\begin{enumerate}
\item $\beta _{\beta _{y}(x*y)} \circ \beta _{y}=\beta _{\beta _{x}(y)} \circ \beta _{x}$ for every $x,y\in Q$, and
\item the map defined by $y\mapsto \beta _{y}(y)$ is a bijection of $Q$.
\end{enumerate}
\end{definition}

By \cite{EH}, every biquandle structure defines a biquandle, and every biquandle arises as a biquandle structure on its associated quandle. 
\begin{Theorem}\cite[Theorem 3.2]{EH} \label{th1} Let $\{\beta _{y}\, |\, y\in Q\}$ be a biquandle structure on a quandle $(Q,*)$. Define two binary operations on $Q$ by $x\dow y=\beta _{y}(x*y)$ and $x\up y=\beta _{y}(x)$ for every $x,y\in Q$. Then $X=(Q,\dow ,\up )$ is a biquandle and $\mathcal{Q}(X)=(Q,*)$. 
\end{Theorem}
\begin{Theorem}\cite[Theorem 3.4]{EH} \label{th2} Let $X=(Q,\dow ,\up )$ be a biquandle and let $\mathcal{Q}(X)=(Q,*)$ be its associated quandle. Then the family of maps $\{\beta _{y}\, |\, y\in Q\}$ is a biquandle structure on $\mathcal{Q}(X)$. 
\end{Theorem}

Nonisomorphic biquandle structures on quandles of order 2 and 3 are listed below. We follow the standard notation of denoting the elements of a finite quandle $(Q,*)$ of order $n$ by numbers $1, 2, \ldots, n$ and its operation table by an $n \times n$ matrix whose $(i, j)$th entry is $i*j,$ see \cite{SN}. A biquandle structure $\{\beta _{y}\, |\, y\in Q\}\subset Aut(Q)$ on such a quandle will be represented by the $n$-tuple $(\beta _{1},\ldots ,\beta _{n})$, where the automorphism $\beta _{i}$ is written as an element of the symmetric group $S_{n}$ in disjoint cycle notation. All computations were performed using \verb|Python|.

\begin{example}
There exists one quandle of order two, namely the trivial quandle with operation table $\begin{bmatrix} 
1 & 1 \\
2 & 2 
\end{bmatrix}$. On this quandle we may impose two nonisomorphic biquandle structures: $(\textrm{id}, \textrm{id})$ or $((12),(12))$. 
\end{example}

\begin{example}
\label{order3}
There are three nonisomorphic quandles of order three. 
\begin{itemize}
\item[(a)] On the trivial quandle with operation table $\begin{bmatrix} 
1 & 1 & 1\\
2 & 2 & 2\\
3 & 3 & 3\\
\end{bmatrix}$ there are 5 nonisomorphic biquandle structures: $(\textrm{id}, \textrm{id},\textrm{id})$, $(\textrm{id}, \textrm{id},(12))$, $(\textrm{id},(23),(23))$, $((23),(23),(23))$ and $((123),(123),(123))$. 
\item[(b)] On the quandle with operation table $\begin{bmatrix} 
1 & 1 & 1\\
3 & 2 & 2\\
2 & 3 & 3\\
\end{bmatrix}$, there are 4 nonisomorphic biquandle structures: $(\textrm{id}, \textrm{id},\textrm{id})$, $(\textrm{id},(23),(23))$, $((23), \textrm{id},\textrm{id})$ and $((23),(23),(23))$.    We note that this quandle is not affine.
\item[(c)] On the quandle with operation table $\begin{bmatrix} 
1 & 3 & 2\\
3 & 2 & 1\\
2 & 1 & 3\\
\end{bmatrix}$, there are 6 nonisomorphic biquandle structures: $(\textrm{id}, \textrm{id},\textrm{id})$, $(\textrm{id},(123),(132))$, $((23),(23),(23))$, $((23),(13),(12))$, $((12),(23),(13))$ and $((123),(123),(123))$.   We remark that this quandle is affine.  
\end{itemize}
\end{example} 
A biquandle structure $\{\beta _{y}\, |\, y\in Q\}$ is called \textbf{constant} if $\beta _{y}=\beta _{z}$ for every $y,z\in Q$. By \cite[Corollary 3.8]{EH}, the number of nonisomorphic constant biquandle structures on a quandle $Q$ is the number of conjugacy classes of $Aut(Q)$. The automorphism groups of the quandles in Example \ref{order3} (a) and (c) are isomorphic to $S_{3}$, thus they admit 3 nonisomorphic constant biquandle structures. The automorphism group of the quandle in (b) is $\ZZ _{2}$, which admits only two nonisomorphic constant biquandle structures. \\

Certain properties of biquandles are inherited from their associated quandles while others are not; we discuss examples of both in the remainder of this Section.  

\begin{lemma}\label{lemma1} In a biquandle $X$, the equality $x\dow y=x\up y$ holds for every $x,y\in X$ if any only if $\q (X)$ is a trivial quandle. 
\end{lemma}
\begin{proof} By Theorems \ref{th1} and \ref{th2}, the biquandle operations are given by $x\dow y=\beta _{y}(x*y)$ and $x\up y=\beta _{y}(x)$ for any $x,y\in X$. Since $\beta _{y}$ is a bijection for every $y\in X$, the equivalence follows. 
\end{proof}

Recall that a quandle $Q$ is \textbf{connected} if for any $x,y\in Q$ there exist elements $z_{1},\ldots ,z_{n}\in Q$ and $\epsilon _{1},\ldots ,\epsilon _{n}\in \{1,-1\}$, such that $y=x*^{\epsilon _{1}}z_{1}*^{\epsilon _{2}}\ldots *^{\epsilon _{n}}z_{n}$.   For example, the quandle in Example \ref{order3} (c) is connected.  

\begin{definition}\label{def4} Given a biquandle $X$, consider the equivalence relation $\sim _{c}$ generated by $x\sim _{c}x\dow y$ and $x\sim _{c}x\up y$ for every $x,y\in X$. The equivalence classes are called \textbf{connected components}, and the biquandle is called \textbf{connected} if there is only one class. 
\end{definition}

\begin{proposition} \label{prop1} If $Q$ is a connected quandle, then for every biquandle structure on $Q$ the induced biquandle is also connected. 
\end{proposition}
\begin{proof}Suppose $\{\beta _{y}\, |\, y\in Q\}\subset Aut(Q)$ is a biquandle structure on $Q$.  We denote the induced biquandle by $B=(Q,\dow ,\up )$. Choose any $x,y\in B$. Since $Q$ is a connected quandle, there exist $z_{1},\ldots ,z_{n}\in Q$ and $\epsilon _{1},\ldots ,\epsilon _{n}\in \{-1,1\}$, such that $y=x*^{\epsilon _{1}}z_{1}*^{\epsilon _{2}}\ldots *^{\epsilon _{n}}z_{n}$. We prove that $x\sim _{c}y$ by induction on $n$. If $n=1$, it follows that either $y\up z_{1}=x\dow z_{1}$ (when $\epsilon _{1}=1$) or $y\dow z_{1}=x\up z_{1}$ (when $\epsilon _{1}=-1$), and thus $x\sim _{c}y$. Now suppose that $x\sim x*^{\epsilon _{1}}z_{1}*^{\epsilon _{2}}\ldots *^{\epsilon _{n-1}}z_{n-1}$ for some $n$. Denoting $w=x*^{\epsilon _{1}}z_{1}*^{\epsilon _{2}}\ldots *^{\epsilon _{n-1}}z_{n-1}$, we obtain $y=w*^{\epsilon _{n}}z_{n}$ and it follows that either $y\up z_{n}=w\dow z_{n}$ or $y\dow z_{n}=w\up z_{n}$, which implies $y\sim _{c}w$ and thus $x\sim _{c}y$. 
\end{proof}

\begin{definition} \label{def5} A quandle $Q$ is called \textbf{medial} if the equality $(x*y)*(z*w)=(x*z)*(y*w)$ holds for every $x,y,z,w\in Q$.  (All three quandles in Example \ref{order3} are medial.)  A biquandle $X$ is called \textbf{medial} if the equalities 
\begin{xalignat*}{1}
& (x\dow y)\dow (z\dow w)=(x\dow z)\dow (y\dow w)\\
& (x\dow y)\up (z\dow w)=(x\up z)\dow (y\up w)\textrm{ and }\\
& (x\up y)\up (z\up w)=(x\up z)\up (y\up w)
\end{xalignat*} hold for every $x,y,z,w\in X$. 
\end{definition}

\begin{lemma} \label{lemma2} If $Q$ is a medial quandle, then for every constant biquandle structure on $Q$, the induced biquandle is also medial. 
\end{lemma}
\begin{proof} Let $Q$ be a medial quandle and $f\in Aut(Q)$. We denote the biquandle induced by the constant biquandle structure $\{f\}$ on $Q$ by $B=(Q,\dow ,\up ).$  Using the mediality of $Q$, the computations:\begin{xalignat*}{1}
& (x\dow y)\dow (z\dow w)=f(f(x*y)*f(z*w))=f^{2}((x*z)*(y*w))=(x\dow z)\dow (y\dow w)\\
& (x\dow y)\up (z\dow w)=f(f(x*y))=f(f(x)*f(y))=(x\up z)\dow (y\up w) \, and \\
&  (x\up y)\up (z\up w)=f(f(x))=(x\up z)\up (y\up w)
\end{xalignat*} imply that $B$ is a medial biquandle. 
\end{proof}

Commutativity is a possible, but not very common property of quandles.  As the following result shows, a commutative quandle cannot be associated to a commutative biquandle.  

\begin{lemma} \label{lemma3} Let $Q$ be a commutative quandle of order $\geq 2$. Then there exists no commutative biquandle $X$ with $\q (X)=Q$. 
\end{lemma} 
\begin{proof} Let $(X,\dow ,\up )$ be a biquandle given by a biquandle structure $\{\beta _{x}\, |\, x\in Q\}\subset Aut(Q)$ on a commutative quandle $Q$. Suppose $X$ is commutative. Then the equations $x\up y=y\up x$ imply that $\beta _{y}(x)=\beta _{x}(y)$ for every $x,y\in X$. Moreover, the equality $x\dow y=y\dow x$ implies that $\beta _{y}(x)*\beta _{y}(y)=\beta _{y}(x*y)=\beta _{x}(y*x)=\beta _{x}(y)*\beta _{x}(x)$. Since $Q$ is commutative, it follows that $\beta _{y}(y)*\beta _{y}(x)=\beta _{x}(x)*\beta _{x}(y)$ and therefore $\beta _{y}(y)=\beta _{x}(x)$ for every $x,y\in X$. Then (2) of Definition \ref{def3} implies that $X$ is of order $\leq 1$.  
\end{proof}

\begin{lemma} \label{lemma4} Let $X$ be a commutative biquandle given by a biquandle structure $\{\beta _{x}\, |\, x\in Q\}$ on a quandle $Q$. Then the automorphism $\beta _{x}\beta _{y}^{-1}\in Aut(Q)$ is of order 2 for every $x\neq y\in Q$. 
\end{lemma}
\begin{proof}The equation $x\up y=y\up x$ implies that $\beta _{y}(x)=\beta _{x}(y)$ for every $x,y\in X$. Moreover, by $x\dow y=y\dow x$ we have that $\beta _{y}(x*y)=\beta _{x}(y*x)$ and by (1) of Definition \ref{def3} it follows that $\beta _{x}\beta _{y}^{-1}=\beta _{y}\beta _{x}^{-1}=(\beta _{x}\beta _{y}^{-1})^{-1}$.  Therefore $(\beta _{x}\beta _{y}^{-1})^{2}=\textrm{id}$. If $\beta _{a}\beta _{b}^{-1}=\textrm{id}$ for some $a,b\in X$, then the equations $\beta _{a}(a)=\beta _{b}(a)=\beta _{a}(b)$ imply that $a=b$. 
\end{proof}

\section{Coloring invariants of links}
\label{sec3}

An important motivation behind the study of quandle-like structures lies in their natural connection with knot theory. In this Section we investigate the relationship between the quandle and biquandle coloring invariant.\\
 
Let $D_{L}$ be an oriented link diagram of a (classical or virtual) link $L$. We denote the set of arcs and the set of crossings of $D_{L}$ by by $A(D_{L})$ and $C(D_{L})$, respectively. Figure \ref{fig:crossingQ} depicts the \textbf{quandle crossing relation} at a crossing of the diagram $D_{L}$.

\begin{figure}[h]
\labellist
\normalsize \hair 2pt
\pinlabel $x$ at 10 150
\pinlabel $y$ at 10 40
\pinlabel $x*y$ at 200 40
\endlabellist
\begin{center}
\includegraphics[scale=0.30]{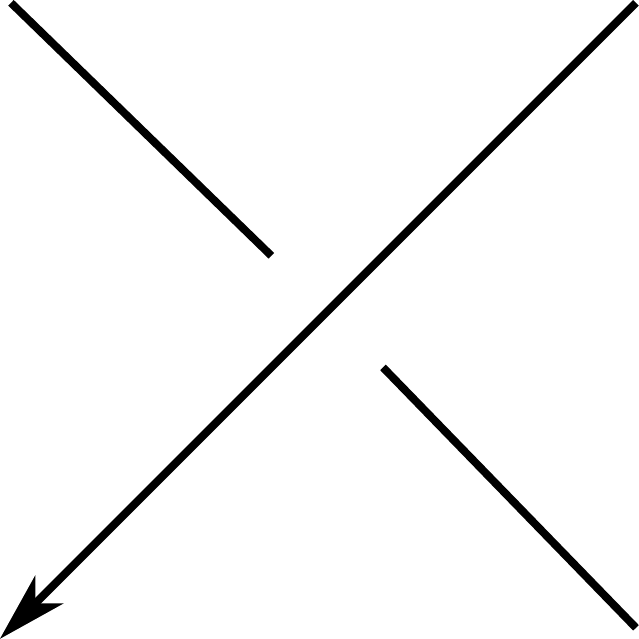}
\caption{The quandle crossing relation}
\label{fig:crossingQ}
\end{center}
\end{figure}

The \textbf{fundamental quandle} of the link $L$ is the quandle $Q(L)$ given by the quandle presentation $$\left \langle A(D_{L})\, |\, \textrm{quandle crossing relation at every $c\in C(D_{L})$}\right \rangle$$ It is easy to see that two diagrams of the same link yield equivalent presentations and thus the fundamental quandle defines a link invariant. For more details, we refer the reader to \cite{FR}. \\

Considering the link diagram $D_{L}$ as a 4-valent graph, every arc is divided into two semiarcs that are incident at a vertex of the graph. We denote the set of semiarcs of $D_{L}$ by $S(D_{L}).$ Figure \ref{fig:crossingB} depicts the \textbf{biquandle crossing relations} at a (positive or negative) crossing of the diagram $D_{L}$. The \textbf{fundamental biquandle} of the link $L$ is the biquandle $B(L)$ given by the biquandle presentation $$\left \langle S(D_{L})\, |\, \textrm{biquandle crossing relations for every $c\in C(D_{L})$}\right \rangle$$ It is well-known that the fundamental biquandle of a classical or virtual link does not depend on the choice of a particular link diagram and thus defines a link invariant \cite{FJK}.

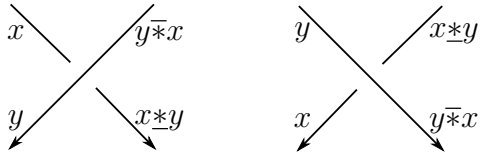
\begin{figure}[h]
\labellist
\normalsize \hair 2pt
\pinlabel $x$ at 10 150
\pinlabel $y$ at 10 40
\pinlabel $x\dow y$ at 190 40
\pinlabel $y\up x$ at 190 150
\pinlabel $y$ at 370 150
\pinlabel $x$ at 370 40
\pinlabel $y\up x$ at 560 40
\pinlabel $x\dow y$ at 560 150
\endlabellist
\begin{center}
\includegraphics[scale=0.30]{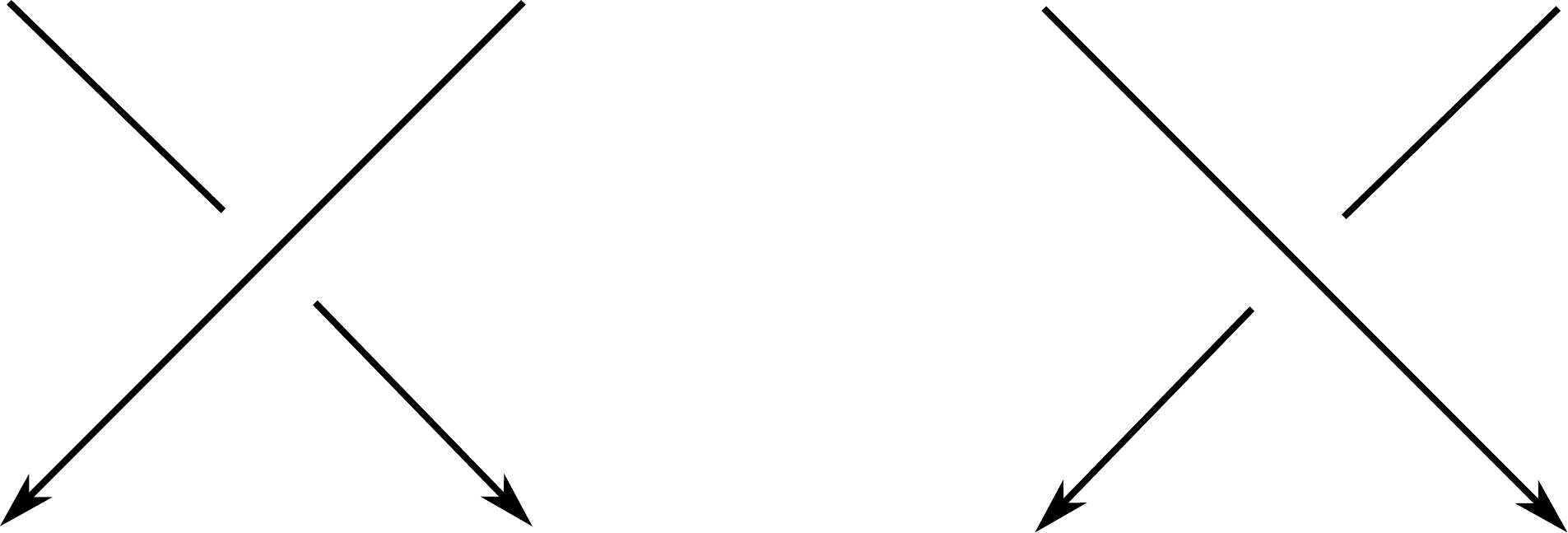}
\caption{Biquandle crossing relations}
\label{fig:crossingB}
\end{center}
\end{figure}
Fundamental (bi)quandles of links are often compared by representations into finite (bi)quandles. We denote the set of quandle (respectively biquandle) homomorphisms between the quandles (respectively biquandles) $X$ and $Y$ by 
$\hoq (X,Y)$ (respectively $\ho (X,Y)$).
\begin{definition} \label{def6} Let $Y$ be a finite quandle. The cardinality $\Phi ^{Y}_{Q}(L)=|\hoq (Q(L),Y)|$ is called the \textbf{quandle coloring invariant} of the link $L$ with respect to $Y$. For a finite biquandle $Z$, the cardinality $\Phi ^{Z}_{B}(L)=|\ho (B(L),Z)|$ is called the \textbf{biquandle coloring invariant} of the link $L$ with respect to $Z$. 
\end{definition}

\begin{example} \label{ex1} Let $Y$ be the quandle of order 4 with operation table $\begin{bmatrix} 
1 & 3 & 4 & 2\\
4 & 2 & 1 & 3\\
2 & 4 & 3 & 1\\
3 & 1 & 2 & 4\\
\end{bmatrix}$. This quandle $Y$ admits 9 nonisomorphic biquandle structures $s_{1},\ldots ,s_{9}$. We will denote the biquandle corresponding to the biquandle structure $s_{i}$ by $Y_i$.  The coloring invariants of some knots with respect to $Y$ and $Y_i$ are listed in the table below. We use the standard knot enumeration from the Knot atlas \cite{Katlas}.
\begin{center}
 \begin{tabular}{| m{2cm} | m{2cm} | m{5cm} |}  
 \hline  \vspace{3 mm}
 Knot & $\Phi _{Q}^{Y}$ & $(\Phi _{B}^{Y_{1}},\ldots ,\Phi _{B}^{Y_{9}})$ \\ [0.5ex] 
 \hline\hline
 $4_1$ & $16$ & $(16,16,4,4,4,4,0,5,4)$\\
 \hline
 $5_1$ & $4$ & $(4,4,4,4,1,1,0,2,0)$\\
 \hline
 $5_2$ & $4$ & $(4, 4, 4, 4, 4, 4, 4, 5, 4)$\\
 \hline
 $6_1$ & $4$ & $(4, 4, 4, 4, 1, 1, 0, 3, 0)$\\
 \hline
 $6_2$ & $4$ & $(4,4,4,4,4,4,4,4,4)$\\
 \hline
 $6_3$ & $4$ & $(4,4,4,4,4,4,4,5,4)$\\
 \hline
\end{tabular}
\end{center}
Observe that the quandle coloring invariant with respect to $Y$ takes the same value for nearly all knots in the table. Also, the biquandle coloring invariant with respect to any one of the biquandles $Y_1,\ldots ,Y_6$ is not very effective in distinguishing knots. The tuple of invariants $(\Phi _{B}^{Y_{1}},\ldots ,\Phi _{B}^{Y_{9}})$, however, is able to distinguish all but two of the knots under consideration.
\end{example}

\begin{example} \label{ex2} Let $Y$ again be the quandle of order 4 from Example \ref{ex1}. The coloring invariants of all 3-crossing virtual knots with respect to $Y$ and $Y_i$ are listed in the table below. The knot enumeration is taken from the Table of Virtual Knots \cite{VKnot}.
\begin{center}
 \begin{tabular}{| m{2cm} | m{2cm} | m{6cm} |}  
 \hline  \vspace{3 mm}
 Virtual knot & $\Phi _{Q}^{Y}$ & $(\Phi _{B}^{Y_{1}},\ldots ,\Phi _{B}^{Y_{9}})$ \\ [0.5ex] 
 \hline\hline
 $3_1$ & $4$ & $(4,4,4,4,1,1,4,6,0)$\\
 \hline
 $3_2$ & $4$ & $(4,4,4,4,1,1,0,4,0)$\\
 \hline
 $3_3$ & $4$ & $(4,4,4,4,4,4,0,3,0)$\\
 \hline
 $3_4$ & $4$ & $(4,4,4,4,1,1,0, 3, 0)$\\
 \hline
 $3_5$ & $4$ & $(4,4,4,4,1,1,4,4,4)$\\
 \hline
$3_6$ & $16$ & $(16,16,16,16,16,16,16,16,16)$\\
\hline
$3_7$ & $4$ & $(4,4,4,4,1,1,4,12,4)$\\
\hline
\end{tabular}
\end{center}
Observe that in contrast with the ordinary quandle and biquandle coloring invariants, the tuple of biquandle coloring invariants $(\Phi _{B}^{Y_{1}},\ldots ,\Phi _{B}^{Y_{9}})$ is able to distinguish all virtual knots in the table.  All computations were performed using \verb|Python|. Our code is available for interested readers upon request. 
\end{example}

The above examples indicate how the richness of biquandle structures on a given quandle may improve the strength of (bi)quandle representation invariants. This lays ground for a new coloring invariant.

\begin{definition} \label{def7} Let $Y$ be a finite quandle that admits $k$ nonisomorphic biquandle structures $s_{1},\ldots ,s_{k}$. Denote by $Y_i$ the biquandle correponding to the biquandle structure $s_{i}$ on $Y$. The $k$-tuple $(\Phi ^{Y_{1}}_{B}(L),\ldots ,\Phi ^{Y_{k}}_{B}(L))$ is called the \textbf{biquandle structure coloring invariant} of the link $L$ with respect to the quandle $Y$.  
\end{definition}

It is clear that the biquandle structure coloring invariant represents an enhancement of both the quandle and biquandle coloring invariants, and thus offers a new way of distinguishing links.

\section{Hom - biquandles}
\label{sec4}

In Section \ref{sec3}, we discussed representations of the fundamental (bi)quandle of a link into finite (bi)quandles. The biquandle coloring invariant of a link $L$ is defined by the cardinality of the homomorphism set $\ho (B(L),Y)$ for a finite biquandle $Y$. It turns out that for suitable choices of the target biquandle $Y$, this set allows an additional structure. 

For two biquandles $X$ and $Y,$ we can endow the morphism set $$\ho (X,Y)=\{f\, |\, f\colon X\to Y\textrm{ is a biquandle homomorphism}\}$$ with two operations $\dow ,\up \colon \ho (X,Y)\times \ho (X,Y)\to \ho (X,Y)$ defined by $(f\dow g)(x)=f(x)\dow g(x)$ and $(f\up g)(x)=f(x)\up g(x)$. A natural question arises as to whether these two operations define a biquandle. The following result has already been established in \cite{CR1}. 

\begin{proposition}\label{prop2} Let $X$ and $Y$ be biquandles. If $Y$ is medial, then $(\ho (X,Y),\dow ,\up )$ is a medial biquandle.
\end{proposition}
\begin{proof} Since $Y$ is a biquandle, it is easy to see that the pointwise operations on $\ho (X,Y)$ will always satisfy the first and third biquandle axioms from Definition \ref{def2}.

To show that $(\ho (X,Y), \dow ,\up )$ satisfies the second biquandle axiom, first consider the map $\alpha _{f}\colon \ho (X,Y)\to \ho (X,Y)$ given by $\alpha _{f}(g)=g\dow f$. We need to show that $\alpha _{f}$ is invertible. Let $h\in \ho (X,Y)$. Since $Y$ is a biquandle, for every $x\in X$ there exists a $g(x)\in Y$ such that $h(x)=g(x)\dow f(x)$. This defines a mapping $g\colon X\to Y$. Since $f$ and $h$ are biquandle homomorphisms and $Y$ is medial, we compute
\begin{eqnarray*}
 g(x\dow y)\dow f(x\dow y) = h(x\dow y) = h(x)\dow h(y) &=& \left( g(x)\dow f(x) \right)\dow \left( g(y)\dow f(y) \right) \\ &=& \left (g(x)\dow g(y)\right ) \dow \left (f(x)\dow f(y)\right ) \\
& = & \left (g(x)\dow g(y)\right )\dow f(x\dow y) \\ & \Rightarrow & g(x\dow y)=g(x)\dow g(y)
\end{eqnarray*}
\begin{eqnarray*}
 g(x\up y)\dow f(x\up y) = h(x\up y) = h(x)\up h(y) &=& \left (g(x)\dow f(x)\right )\up \left (g(y)\dow f(y)\right ) \\
 &=& \left (g(x)\up g(y)\right )\dow \left (f(x)\up f(y)\right ) \\
 &=& \left (g(x)\up g(y)\right )\dow f(x\up y) \\ & \Rightarrow & g(x\up y)=g(x)\up g(y)
\end{eqnarray*} 
and it follows that $g$ is a biquandle homomorphism for which $\alpha _{f}(g)=h$. 

Secondly, consider the map $\beta _{f}\colon \ho (X,Y)\to \ho (X,Y)$ given by $\beta _{f}(g)=g\up f$. To show that $\beta _{f}$ is invertible, choose $h\in \ho (X,Y)$. Since $Y$ is a biquandle, for every $x\in X$ there exists $g(x)\in Y$ such that $h(x)=g(x)\up f(x)$. This defines a mapping $g\colon X\to Y$. Using the mediality of $Y$, we compute
\begin{eqnarray*}
 g(x\dow y)\up f(x\dow y) = h(x\dow y) = h(x)\dow h(y)
 &=& \left (g(x)\up f(x)\right )\dow \left (g(y)\up f(y)\right ) \\
 &=&\left (g(x)\dow g(y)\right )\up \left (f(x)\dow f(y)\right ) \\
&=& \left (g(x)\dow g(y)\right )\up f(x\dow y) \\
& \Rightarrow & g(x\dow y)=g(x)\dow g(y) 
\end{eqnarray*}
\begin{eqnarray*}
g(x\up y)\up f(x\up y) = h(x\up y) = h(x)\up h(y) &=& \left (g(x)\up f(x)\right )\up \left (g(y)\up f(y)\right ) \\
&=& \left (g(x)\up g(y)\right )\up \left (f(x)\up f(y)\right ) \\
&=& \left (g(x)\up g(y)\right )\up f(x\up y) \\
& \Rightarrow & g(x\up y) = g(x)\up g(y)
\end{eqnarray*} and thus $g$ is a biquandle homomorphism for which $\beta _{f}(g)=h$. 

Thirdly, consider the map $S\colon \ho (X,Y)\times \ho (X,Y)\to \ho (X,Y)\times \ho (X,Y)$. Let $h,k\in \ho (X,Y)$. Since $Y$ is a biquandle, for any $x\in X$ there exist two elements $f(x),g(x)\in Y$ such that $(h(x),k(x))=S(f(x),g(x))=\left (g(x)\up f(x),f(x)\dow g(x)\right )$. We need to show that the maps $f,g\colon X\to Y$ are biquandle homomorphisms. We have
\begin{xalignat*}{1}
& g(x\dow y)\up f(x\dow y)=h(x\dow y)=h(x)\dow h(y)=(g(x)\up f(x))\dow (g(y)\up f(y))=(g(x)\dow g(y))\up (f(x)\dow f(y))\\
& f(x\dow y)\dow g(x\dow y)=k(x\dow y)=k(x)\dow k(y)=(f(x)\dow g(x))\dow (f(y)\dow g(y))=(f(x)\dow f(y))\dow (g(x)\dow g(y))
\end{xalignat*} The obtained equalities imply that $S\left (f(x\dow y),g(x\dow y)\right )=S\left (f(x)\dow f(y),g(x)\dow g(y)\right )$ and since $S\colon Y\times Y\to Y\times Y$ is invertible, it follows that $f(x\dow y)=f(x)\dow f(y)$ and $g(x\dow y)=g(x)\dow g(y)$. Similarly, the equalities
\begin{xalignat*}{1}
& g(x\up y)\up f(x\up y)=h(x\up y)=h(x)\up h(y)=(g(x)\up f(x))\up (g(y)\up f(y))=(g(x)\up g(y))\up (f(x)\up f(y))\\
& f(x\up y)\dow g(x\up y)=k(x\up y)=k(x)\up k(y)=(f(x)\dow g(x))\up (f(y)\dow g(y))=(f(x)\up f(y))\dow (g(x)\up g(y))
\end{xalignat*} imply that $S\left (f(x\up y),g(x\up y)\right )=S\left (f(x)\up f(y),g(x)\up g(y)\right )$, and thus $f(x\up y)=f(x)\up f(y)$ and $g(x\up y)=g(x)\up g(y)$. We have therefore shown that $f,g\in \ho (X,Y)$ and $S(f,g)=(h,k)$. 

Thus, the mediality of $Y$ implies the mediality of $\ho (X,Y)$ with the pointwise operations.
\end{proof}

\begin{definition}For biquandles $X$ and $Y$ where $Y$ is medial, the biquandle $(\ho (X,Y),\dow ,\up )$ will be called the \textbf{Hom-biquandle} and denoted by $\Ho (X,Y)$.
\end{definition}

Similarly, if $Q_1$ is a quandle and $Q_2$ is a medial quandle, the set of quandle homomorphisms $$\hoq (Q_1,Q_2)=\{f\colon Q_1\to Q_2\, |\, \textrm{$f$ is a quandle homomorphism}\}$$ forms a quandle with the operation $(f*g)(x)=f(x)*g(x)$ for every $x\in Q_1$ \cite{CR1}. This quandle is called the \textbf{Hom-quandle} and will be denoted by $\Hoq (X,Y)$. Structure and properties of Hom-quandles were studied in \cite{CR1},\cite{CR2}. \\

As we have seen, every biquandle arises by imposing a biquandle structure on its associated quandle. A natural question is:  What is the associated quandle of the Hom-biquandle $\Ho (X,Y)$? As a set, $\q (\Ho (X,Y))=\ho (X,Y)$ is the set of biquandle homomorphisms. The quandle operation is given by $$(f*g)(x)=((f\dow g)\up ^{-1}g)(x)=(f(x)\dow g(x))\up ^{-1}g(x)=f(x)*g(x)$$ where $*$ on the right-hand side denotes the quandle operation on $\q (Y)$. It follows that $\q (\Ho (X,Y))$ is a subset of  $\Hoq (\q (X),\q (Y))$ -- the associated quandle of the Hom-biquandle is a subquandle of the Hom-quandle of associated quandles. The characterization of this subquandle is given below. 

\begin{proposition} \label{prop3} Let $(Q,*)$ and $(K,\tr )$ be quandles. Suppose $\{\alpha _{x}|\, x\in Q\}\subset Aut(Q)$ is a biquandle structure defining a biquandle $X$ and $\{\beta _{y}|\, y\in K\}\subset Aut(K)$ is a biquandle structure defining a biquandle $Y$. A quandle homomorphism $f\colon Q\to K$ lifts to a biquandle homomorphism $\widetilde{f}\colon X\to Y$ if and only if $f\alpha _{x}=\beta _{f(x)}f$ for every $x\in Q$. 
\end{proposition}
\begin{proof}For any $z,x\in X$ we have 
\begin{xalignat*}{1}
& \widetilde{f}(z\dow x)=f(\alpha _{x}(z*x))\quad \textrm{and} \quad \widetilde{f}(z)\dow \widetilde{f}(x)=\beta _{f(x)}(f(z)\tr f(x))=\beta _{f(x)}(f(z*x))\\
& \widetilde{f}(z\up x)=f(\alpha _{x}(z))\quad \textrm{and}\quad \widetilde{f}(z)\up \widetilde{f}(x)=\beta _{f(x)}(f(z))
\end{xalignat*}
\end{proof}

\begin{proposition}\label{prop4} Let $X$ be a biquandle defined by a biquandle structure $\{\alpha _{x}|\, x\in \q (X)\}$. Let $Y$ be a medial biquandle, defined by a biquandle structure $\{\beta _{y}|\, y\in \q (Y)\}$.  Then $\q (\Ho (X,Y))$ is the subquandle $$\left \{f\in \Hoq (\q (X),\q (Y))\, |\, f\alpha _{x}=\beta _{f(x)}f \textrm{ for every $x\in \q (X)$}\right \}\;$$ The biquandle structure of $\Ho (X,Y)$  is given by $\{\beta _{g}^{*}\,|\, g\in \ho (X,Y)\}$. 
\end{proposition}
\begin{proof} The first statement follows directly from Proposition \ref{prop3} together with the discussion preceding the Proposition. For the second statement, observe that 
\begin{xalignat*}{1}
& (f\dow g)(x)=f(x)\dow g(x)=\beta _{g(x)}(f(x)*g(x)) \; \textrm{and}\\
& (f\up g)(x)=f(x)\up g(x)=\beta _{g(x)}(f(x))
\end{xalignat*} for any $f,g\in \ho (X,Y)$ and $x\in X$. 
\end{proof}

\begin{example} \label{ex3} Consider nonisomorphic biquandles of order 3, listed in Example \ref{order3}. Table \ref{tab1} lists cardinalities of the Hom-biquandle for every pair of biquandles of order 3. We denote by $A_{i}$ (respectively $B_{i}$ or $C_{i}$) the biquandle, corresponding to the $i$-th biquandle structure in Example \ref{order3} (a) (resp. (b) or (c)). Compare this to Table \ref{tab2} that lists cardinalities of the Hom-quandle of associated quandles. 

\begin{table}[h!]
\begin{center}
 \begin{tabular}{| m{1cm} || c | c | c | c | c || c | c | c | c || c | c | c | c | c | c ||}  
 \hline 
 $X\backslash Y$ & $A_1$ & $A_2$ & $A_3$ & $A_4$ & $A_5$ & $B_1$ & $B_2$ & $B_3$ & $B_4$ & $C_1$ & $C_2$ & $C_3$ & $C_4$ & $C_5$ & $C_6$ \\ [0.5ex] 
 \hline
 $A_1$ & $27$ & $17$ & $9$ & $9$ & $0$ & $9$ & $1$ & $9$ & $1$ & $3$ & $1$ & $1$ & $3$ & $0$ & $0$ \\
 \hline
 $A_2$ & $9$ & $9$ & $3$ & $3$ & $0$ & $5$ & $1$ & $5$ & $1$ & $3$ & $1$ & $1$ & $3$ & $0$ & $0$ \\
 \hline
 $A_3$ & $27$ & $17$ & $9$ & $9$ & $0$ & $9$ & $1$ & $9$ & $1$ & $3$ & $1$ & $1$ & $3$ & $0$ & $0$ \\
 \hline
 $A_4$ & $27$ & $17$ & $9$ & $9$ & $0$ & $9$ & $1$ & $9$ & $1$ & $3$ & $1$ & $1$ & $3$ & $0$ & $0$ \\
 \hline
 $A_5$ & $9$ & $7$ & $7$ & $9$ & $9$ & $5$ & $5$ & $5$ & $5$ & $3$ & $1$ & $1$ & $3$ & $0$ & $0$ \\
 \hline
 \hline
 $B_1$ & $9$ & $7$ & $3$ & $3$ & $0$ & $7$ & $3$ & $7$ & $3$ & $3$ & $1$ & $1$ & $3$ & $0$ & $0$ \\
 \hline
 $B_2$ & $9$ & $7$ & $3$ & $3$ & $0$ & $7$ & $3$ & $7$ & $3$ & $3$ & $1$ & $1$ & $3$ & $0$ & $0$ \\
 \hline
 $B_3$ & $9$ & $7$ & $3$ & $3$ & $0$ & $7$ & $3$ & $7$ & $3$ & $3$ & $1$ & $1$ & $3$ & $0$ & $0$ \\
 \hline
 $B_4$ & $9$ & $7$ & $3$ & $3$ & $0$ & $7$ & $3$ & $7$ & $3$ & $3$ & $1$ & $1$ & $3$ & $0$ & $0$ \\
 \hline
 \hline
 $C_1$ & $3$ & $3$ & $1$ & $1$ & $0$ & $3$ & $1$ & $3$ & $1$ & $9$ & $1$ & $3$ & $3$ & $0$ & $0$ \\
 \hline
 $C_2$ & $3$ & $3$ & $1$ & $1$ & $0$ & $3$ & $1$ & $3$ & $1$ & $3$ & $3$ & $1$ & $3$ & $0$ & $0$ \\
 \hline
 $C_3$ & $3$ & $3$ & $1$ & $1$ & $0$ & $3$ & $1$ & $3$ & $1$ & $9$ & $1$ & $3$ & $3$ & $0$ & $0$ \\
 \hline
 $C_4$ & $3$ & $3$ & $1$ & $1$ & $0$ & $3$ & $1$ & $3$ & $1$ & $3$ & $1$ & $1$ & $9$ & $0$ & $0$ \\
 \hline
 $C_5$ & $3$ & $3$ & $1$ & $1$ & $0$ & $3$ & $1$ & $3$ & $1$ & $3$ & $1$ & $1$ & $3$ & $3$ & $0$ \\
 \hline
 $C_6$ & $3$ & $3$ & $1$ & $1$ & $0$ & $3$ & $1$ & $3$ & $1$ & $3$ & $1$ & $3$ & $3$ & $0$ & $3$ \\
 \hline
\end{tabular}
\bigskip
\caption{Cardinalities of $\ho (X,Y)$.\label{tab1}}
\end{center}
\end{table}
Table \ref{tab1} shows that the Hom-biquandles $\Ho (B_2,B_2)$ and $\Ho (B_2,A_3)$ share the same order. Calculation reveals that $\ho (B_2,B_2)=\{(1,1,1),(1,2,3),(1,3,2)\}$ and $\ho (B_2,A_3)=\{(1,1,1),(1,2,2),(1,3,3)\}$. Taking into account the biquandle operations, it is easy to check that $\Ho (B_2,B_2)\cong B_2$ and $\Ho (B_2,A_3)\cong A_3$, thus the Hom-biquandles are not isomorphic. 
\begin{table}[h!]
\begin{center}
 \begin{tabular}{| m{1.5cm} || c || c || c |}  
 \hline  
 $Q_{1}\backslash Q_{2}$ & $\q (A_{i})$ & $\q (B_{i})$ & $\q (C_{i})$ \\ [0.5ex] 
 \hline
  \hline
 $\q (A_{i})$ & $27$ & $9$ & $3$ \\
 \hline
  \hline
 $\q (B_{i})$ & $9$ & $7$ & $3$ \\
 \hline
  \hline
 $\q (C_{i})$ & $3$ & $3$ & $9$ \\
 \hline
\end{tabular}
\bigskip
\caption{Cardinalities of $\hoq (Q_{1},Q_{2})$.\label{tab2}}
\end{center}
\end{table}

\end{example}

By Proposition \ref{prop4}, the associated quandle of $\Ho (X,Y)$ depends on the biquandle structures of both biquandles $X$ and $Y$. The biquandle structure of the Hom-biquandle, however, is determined solely by the biquandle structure of $Y$. This fact is reflected in Lemma \ref{lemma5} and Proposition \ref{prop5}. Recall that a biquandle $Y$ is called \textbf{involutory} if the equalities 
\begin{xalignat*}{3}
& x\dow (y\up x)=x\dow y\,, & x\up (y\dow x)=x\up y\,, \quad \quad & (x\dow y)\dow y=x\quad \quad \textrm{ and } & (x\up y )\up y=x
\end{xalignat*} hold for every $x,y\in Y$. 

\begin{lemma}\label{lemma5}Let $X$ be a biquandle and let $Y$ be a medial biquandle. \begin{enumerate}
\item[(a)] If $Y$ is involutory, then $\Ho (X,Y)$ is also involutory. 
\item[(b)] If $Y$ is commutative, then $\Ho (X,Y)$ is also commutative. 
\end{enumerate}
\end{lemma}
\begin{proof} This follows from a straightforward computation. 
\end{proof}

A biquandle $X$ is called a \textbf{constant action biquandle} if $x\dow y=x\up y=\sigma (x)$ for some bijection $\sigma \colon X\to X$.

\begin{lemma}\label{lemma6}Any constant action biquandle is medial. 
\end{lemma}
\begin{proof} Let $X$ be a constant action biquandle in which $x\dow y=x\up y=\sigma (x)$ for some bijection $\sigma \colon X\to X$. By Lemma \ref{lemma1}, its associated quandle $\q (X)$ is trivial and thus medial. Since $X$ is defined by the constant biquandle structure $\{\sigma \}\subset Aut(\q (X))$, it is medial by Lemma \ref{lemma2}. 
\end{proof}  

\begin{proposition}\label{prop5} Let $X$ be a biquandle. If  $Y$ is a constant action biquandle, then $\Ho (X,Y)$ is a constant action biquandle. 
\end{proposition}
\begin{proof}Let $Y$ be a constant action biquandle. There exists a bijection $\sigma \colon Y\to Y$ such that $x\dow y=x\up y=\sigma (x)$ for every $x,y\in Y$. It follows that $$(f\dow g)(x)=f(x)\dow g(x)=f(x)\up g(x)=(f\up g)(x)=\sigma (f(x))$$ for every $f,g\in \Ho (X,Y)$ and every $x\in X$. Define a map $\sigma ^{*}\colon \Ho (X,Y)\to \Ho (X,Y)$ by $\sigma ^{*}(f)=\sigma \circ f$. Since $\sigma $ is injective, it follows that $\sigma ^{*}$ is injective. 

It remains to show that $\sigma ^{*}$ is surjective. Let $g\in \Ho (X,Y)$. Since $\sigma $ is surjective, for every $x\in X$ there exists a $\psi (x)\in Y$ such that $\sigma (\psi (x))=g(x)$. This defines a function $\psi \colon X\to Y$. We need to show that $\psi $ is a biquandle homomorphism. We have 
\begin{xalignat*}{1}
& \psi (x_1\dow x_2)=\sigma ^{-1}(g(x_1\dow x_2))=\sigma ^{-1}(g(x_1)\dow g(x_2))=\sigma ^{-1}(\sigma (g(x_{1})))=g(x_{1}) \; \textrm{and}\\
& \psi (x_1)\dow \psi (x_2)=(\sigma ^{-1}\circ g)(x_1)\dow (\sigma ^{-1}\circ g)(x_{2})=\sigma ^{-1}(g(x_{1}))\dow \sigma ^{-1}(g(x_{2}))=\sigma (\sigma ^{-1}(g(x_1)))=g(x_1)
\end{xalignat*} Therefore $\psi (x_1\dow x_2)=\psi (x_1)\dow \psi (x_2)$, and an analogous calculation shows that $\psi (x_1\up x_2)=\psi (x_1)\up \psi (x_2)$. We have thus found a biquandle homomorphism $\psi \colon X\to Y$ for which $g=\sigma ^{*}(\psi )$. Therefore $\sigma ^{*}$ is a bijection on $\Ho (X,Y)$ and $f\dow g=f\up g=\sigma ^{*}(f)$ for every $f,g\in \Ho (X,Y)$, which shows that $\Ho (X,Y)$ is a constant action biquandle. 
\end{proof}

\begin{proposition}\label{prop6}
$\ho (-,Z)$ is a functor from the category of biquandles to the category of medial biquandles for any medial biquandle $Z$. $\ho (A,-)$ is an endofunctor of the category of medial biquandles for any biquandle $A$. 
\end{proposition}
\begin{proof}Let $Z$ be a medial biquandle. For any biquandles $X$ and $Y,$ $\Ho (X,Z)$ and $\Ho (Y,Z)$ are biquandles by Proposition \ref{prop2}. For a biquandle homomorphism $h\colon X\to Y$, the map $h_{*}\colon \Ho (Y,Z)\to \Ho (X,Z)$ is given by $h_{*}(f)=f\circ h$. To see that $h_{*}$ is a biquandle homomorphism, we compute
\begin{xalignat*}{1}
& h_{*}(f\dow g)(x)=((f\dow g)\circ h)(x)=f(h(x))\dow g(h(x))=(h_{*}(f)\dow h_{*}(g))(x)\;,
\end{xalignat*} and similarly for the other operation. 

Let $A$ be a biquandle. For any medial biquandles $X$ and $Y$, $\Ho (A,X)$ and $\Ho (A,Y)$ are medial biquandles by Proposition \ref{prop2}. For a biquandle homomorphism $h\colon X\to Y$, the map $h^{*}\colon \Ho (A,X)\to \Ho (A,Y)$ is given by $h^{*}(f)=h\circ f$. To check that $h^{*}$ is a biquandle homomorphism, we compute
\begin{xalignat*}{1}
& h^{*}(f\dow g)(x)=(h\circ (f\dow g))(x)=h(f(x))\dow h(g(x))=(h^{*}(f)\dow h^{*}(g))(x)\;,
\end{xalignat*} and similarly for the other operation. 
\end{proof}

For a finitely generated biquandle $X$ and a medial biquandle $Y$, biquandles $Y$ and $\Ho (X,Y)$ are also related via subbiquandle inclusions. The following statement generalizes an analogous result about Hom-quandles, see \cite[Theorem 8]{CR1}.

\begin{proposition}\label{prop7} Let $X$ be a finitely generated biquandle and let $Y$ be a medial biquandle. Then $\Ho (X,Y)$ is isomorphic to a subbiquandle of $Y^{k}$, where $k$ is the size of a minimal generating set for $X$. 
\end{proposition}
\begin{proof}Let $\{x_{1},\ldots ,x_{k}\}$ be a minimal generating set for $X$. Define a map $j\colon \Ho (X,Y)\to Y^{k}$ by $j(f)=(f(x_{1}),\ldots ,f(x_{k}))$. 

Every biquandle homomorphism in $\Ho (X,Y)$ is completely determined by its values on the generating set, thus $j$ is injective. For two homomorphisms $f,g\in \Ho (X,Y)$ we have 
\begin{eqnarray*}
 j(f\dow g) = \left ((f\dow g)(x_{1}),\ldots ,(f\dow g)(x_{k})\right ) &=& \left (f(x_1)\dow g(x_1),\ldots ,f(x_k)\dow g(x_k)\right ) \\
&=& \left (f(x_{1}),\ldots ,f(x_{k})\right )\dow \left (g(x_{1}),\ldots ,g(x_k)\right ) \\
&=& j(f)\dow j(g)
\end{eqnarray*}  and similarly for the other operation. It follows that $j$ is a biquandle monomorphism, and thus $\Ho (X,Y)$ is isomorphic to $Im(j)\leq Y^{k}$. 
\end{proof}

\begin{remark}We have shown that $\Ho (X,Y)$ is isomorphic to the subbiquandle $Im(j)$: $$\left \{\left (f(x_{1}),\ldots ,f(x_k)\right )\, |\, f(x_{i}\dow x_{j})=f(x_i)\dow f(x_j)\textrm{ and }f(x_{i}\up x_{j})=f(x_i)\up f(x_j)\, \forall i,j\in \{1,\ldots ,k\}\right \}\;,$$ which is precisely the set of all biquandle colorings of $\{x_{1},\ldots ,x_{k}\}$ by the biquandle $Y$. 
\end{remark}

In the remainder of this Section, we investigate how the source biquandle $X$ of $\Ho (X,Y)$ may be simplified and still yield the same Hom-biquandle. Our results generalize the results about Hom-quandles from \cite{CR2}.

\begin{definition} An equivalence relation $\sim $ on a biquandle $X$ is called a \textbf{congruence} if ($x\sim y$ and $z\sim w$) implies ($x\dow z\sim y\dow w$ and $x\up z\sim y\up w$) for every $x,y,z,w\in X$.
\end{definition}   

 For each congruence on $X$, the quotient set $X/_{\sim }$ forms a \textbf{quotient biquandle} with the induced operations on equivalence classes.
  
 Let $\mathcal{I}$ be a collection of identities on a biquandle $X$. We denote by $Cg(\mathcal{I})$ the minimal congruence such that $a\sim b$ whenever there exist $x_{1},\ldots ,x_{n}$ such that $a=p(x_{1},\ldots ,x_{n})$ and $b=q(x_{1},\ldots ,x_{n})$, where $p=q$ is an identity in $\mathcal{I}$. We call $Cg(\mathcal{I})$ the congruence, generated by $\mathcal{I}$.  

\begin{proposition}\label{prop8}Let $X$ and $Y$ be biquandles and let $\mathcal{I}$ be a set of identities, satisfied by $Y$. Then $\ho (X,Y)\cong \ho (X/Cg(\mathcal{I}),Y)$ as sets. 
\end{proposition}
\begin{proof}Denote by $\pi _{Cg(\mathcal{I})}\colon X\to X/Cg(\mathcal{I})$ the quotient homomorphism. For every $f\in \ho (X,Y)$ and for any identity $p=q$ in $\mathcal{I}$, we have $$f(p(x_{1},\ldots ,x_{n}))=p(f(x_{1},\ldots ,x_{n}))=q(f(x_{1},\ldots ,x_{n}))=f(q(x_{1},\ldots ,x_{n}))\;,$$ therefore $Cg(\mathcal{I})\subseteq Ker(f)$. By the First Isomorphism Theorem, there exists a unique biquandle homomorphism $\widetilde{f}\colon X/Cg(\mathcal{I})\to Y$ such that $\widetilde{f}\circ \pi _{Cg(\mathcal{I})}=f$. We define a map $\phi \colon \ho (X,Y)\to \ho (X/Cg(\mathcal{I}),Y)$ by $\phi (f)=\widetilde{f}$. Then $\phi $ is a bijection with inverse $\psi \colon \ho (X/Cg(\mathcal{I}),Y)\to \ho (X,Y)$, given by $\psi (g)=g\circ \pi _{Cg(\mathcal{I})}$. 
\end{proof}

\begin{definition} A biquandle $X$ is called \textbf{2-reductive} if the equalities 
\begin{xalignat*}{1}
& a\dow (b\dow c)=a\dow b\quad \quad a\up (b\up c)=a\up b\\
& a\dow (b\up c)=a\dow b\quad \quad a\up (b\dow c)=a\up b
\end{xalignat*} are satisfied for every $a,b,c\in X$. 
\end{definition} 

For example, every constant action biquandle is 2-reductive. 

\begin{lemma}\label{lemma7} A 2-reductive biquandle is medial. 
\end{lemma}
\begin{proof} Choose elements $a,b,c$ and $d$ of a 2-reductive biquandle $X$. Using 2-reductiveness and the third biquandle axiom, we compute
\begin{xalignat*}{1}
& (a\dow b)\dow (c\dow d)=(a\dow b)\dow c=(a\dow b)\dow (c\dow b)=(a\dow c)\dow (b\up c)=(a\dow c)\dow b=(a\dow c)\dow (b\dow d)\;,\\
& (a\up b)\up (c\up d)=(a\up b)\up c=(a\up b)\up (c\up b)=(a\up c)\up (b\dow c)=(a\up c)\up b=(a\up c)\up (b\up d)\textrm{ and }\\
& (a\up b)\dow (c\up d)=(a\up b)\dow c=(a\up b)\dow (c\up b)=(a\dow c)\up (b\dow c)=(a\dow c)\up b=(a\dow c)\up (b\dow d)\;.
\end{xalignat*} 
\end{proof}

 In a biquandle $X$, consider the relation 
\begin{xalignat*}{1}
& R=\left \{\left (a\dow (b\dow c),a\dow b\right ),\left (a\up (b\up c),a\up b\right ),\left (a\dow (b\up c),a\dow b\right ),\left (a\up (b\dow c),a\up b\right )\, |\, a,b,c\in X\right \}
\end{xalignat*} and denote the congruence generated by $R$ by $\gamma _{X}$. Relation $\gamma _{X}$ is the smallest congruence such that the quotient $X/\gamma _{X}$ is 2-reductive.

\begin{proposition}\label{prop9} Let $X$ be a biquandle and let $Y$ be a 2-reductive biquandle. Then $X/\gamma _{X}$ is 2-reductive and $\Ho (X,Y)\cong \Ho (X/\gamma _{X},Y)$ as biquandles.  
\end{proposition} 
\begin{proof}Since $Y$ is 2-reductive, by Proposition \ref{prop8} there exists a bijection $\phi \colon \Ho (X,Y)\to \Ho (X/\gamma _{X},Y)$, which is given by $\phi (f)=\widetilde{f}$, where $\widetilde{f}\circ \pi _{\gamma _{X}}=f$. For any $f,g\in \Ho (X,Y)$ we have $\phi (f\dow g)=\widetilde{(f\dow g)}$, such that $\widetilde{(f\dow g)}(\pi _{\gamma _{X}}(x))=(f\dow g)(x)=f(x)\dow g(x)=\widetilde{f}(\pi _{\gamma _{X}}(x))\dow \widetilde{g}(\pi _{\gamma _{X}}(x))=\phi (f)(x)\dow \phi (g)(x)$. A similar calculation shows that $\phi (f\up g)=\phi (f)\up \phi (g)$ and it follows that $\phi $ is a biquandle isomorphism.  
\end{proof}

\section{Directions for Future Investigation}
\label{sec5}

We first wonder whether the analogs of the questions posed in the final section of \cite{CR1} hold for the Hom-biquandle.  That is, what other properties, other than those presented here, does the Hom-biquandle inherit from the source and target biquandles?  Given two connected biquandles, is the Hom-biquandle structure determined by the counting invariant?   

In addition, we seek a relationship between the cardinalities of the source and target biquandles and that of the Hom-biquandle.  In particular, could the notion of 2-reductiveness  lead to finding an analog of Corollary 3.24 in \cite{CR2}, enabling us to count and characterize the Hom-biquandle of a 2-reductive target and arbitrary source?

Finally, when considering a more complicated study of links (e.g.  virtual links), we sometimes must combine two or more different link invariants to obtain a stronger invariant.  What role can the Hom-biquandle play in these situations?

\section*{Acknowledgments}

Eva Horvat was supported by the Slovenian Research Agency grant N1-0083.  Alissa S. Crans was supported by a grant from the Simons Foundation (\#360097, Alissa Crans).

\end{document}